\documentclass[11pt]{amsart}

\usepackage[margin=1in]{geometry}
\usepackage{graphicx}
\usepackage{mathtools}
\usepackage{amssymb,amsfonts,amsthm}
\usepackage{enumerate}
\usepackage{color}
\usepackage[usenames,dvipsnames,svgnames,table]{xcolor}
\usepackage{comment}
\usepackage{multirow}
\usepackage{tikz}
\usepackage{hyperref}
\usepackage{euscript}
\specialcomment{itsapicture}{}{}

\newcommand{\inv}{^{-1}}
\newcommand{\floor}[1]{\left\lfloor #1 \right\rfloor}

\newcommand{\eucal}{\EuScript}
\newcommand{\fr}{\mathfrak}

\newcommand{\ff}{\mathbf{F}}

\newcommand{\pp}{\mathbf{P}}
\newcommand{\qq}{\mathbf{Q}}

\newcommand{\zz}{\mathbf{Z}}
\newcommand{\ring}{{\eucal O}}

\newcommand{\disc}{\operatorname{disc}}

\newcommand{\ind}{\operatorname{ind}}
\newcommand{\gal}{\operatorname{Gal}}

\newcommand{\ord}{\operatorname{ord}}
\newcommand{\period}{\operatorname{per}}
\newcommand{\pper}{\operatorname{pper}}
\newcommand{\per}{\operatorname{per}}

\newcommand{\red}{\operatorname{red}}
\newcommand{\res}{\operatorname{res}}

\newcommand{\wt}{\operatorname{wt}}

\newtheorem{thm}{Theorem}[section]
\newtheorem{prop}[thm]{Proposition}

\newtheorem{lem}[thm]{Lemma}

\theoremstyle{definition}

\newtheorem{example}[thm]{Example}

\theoremstyle{remark}
\newtheorem{remark}[thm]{Remark}

\allowdisplaybreaks

\begin{document}
\title{Discriminants of simplest $3^n$-tic extensions}
\author[t. alden gassert]{T. Alden Gassert}
\email{thomas.gassert@colorado.edu}
\address{University of Colorado, Boulder\\
Campus Box 395\\
Boulder, CO, USA 80309-0395}

\date{\today}

\begin{abstract}
Let $\ell>2$ be a positive integer, $\zeta_\ell$ a primitive $\ell$-th root of unity, and $K$ a number field containing $\zeta_\ell+\zeta_\ell\inv$ but not $\zeta_\ell$. In a recent paper, Chonoles et. al. study iterated towers of number fields over $K$ generated by the generalized Rikuna polynomial, $r_n(x,t;\ell) \in K(t)[x]$. They note that when $K = \qq$, $t \in \{0,1\}$, and $\ell=3$, the only ramified prime in the resulting tower is 3, and they ask under what conditions is the number of ramified primes small. In this paper, we apply a theorem of Gu\`ardia, Montes, and Nart to derive a formula for the discriminant of $\qq(\theta)$ where $\theta$ is a root of $r_n(x,t;3)$, answering the question of Chonoles et. al. in the case $K = \qq$, $\ell=3$, and $t \in \zz$. In the latter half of the paper, we identify some cases where the dynamics of $r_n(x,t;\ell)$ over finite fields yields an explicit description of the decomposition of primes in these iterated extensions.
\end{abstract}

\maketitle

\section{Introduction}
Let $K$ be a number field, $\ring_K$ its ring of integers, and $\varphi$ a rational map of degree at least 2 with coefficients in $\ring_K$. Denote by $\varphi^n(x) = \varphi \circ \varphi^{n-1}(x)$ its $n$-fold composition, and for each iterate write $\varphi^n(x) = p_n(x)/q_n(x)$, where $p_n(x), q_n(x) \in \ring_K[x]$ share no common roots. For any $t \in \ring_K$, consider a sequence of preimages of $t$ under $\varphi$, say $(t=\theta_0, \theta_1,\theta_2, \ldots)$, which satisfies $\varphi(\theta_n) = \theta_{n-1}$. If the selection of $\varphi$ and $t$ are such that $p_n(x) - tq_n(x)$ is monic and irreducible for each $n \ge 1$, then by adjoining the preimages of $t$ to $K$, we obtain an infinite tower of fields
\begin{align*}
K = K_0\subset K_1 \subset K_2 \subset K_3 \cdots,
\end{align*}
where $K_n = K(\theta_n)$. These fields are known as \emph{iterated extensions}, and we use $\{K_n\} :=\{K_n;t\}_{n=0}^\infty$ to denote any tower of iterated extensions over $K$ specialized at $t$.

In recent years, these dynamically generated fields, particularly those coming from the iteration of postcritically finite maps, have been used to produce towers of number fields with interesting and unusual arithmetic qualities. (Recall that a map $\varphi$ is \emph{postcritically finite} if the forward orbit of each of its critical points is finite.) For one, any specialized tower generated by a postcritically finite map is unramified outside a finite set of primes \cite{ahm05,ch12}. Moreover, some families of postcritically finite maps---most notably the power maps, which generate Kummer extensions, the Chebyshev polynomials \cite[Proposition 5.6]{bgn03}, and Latt\`es maps \cite{jr10}---give rise to extensions whose associated Galois groups are relatively small. In general, if $K_{G,n}$ denotes the Galois closure of $K_n/K$, the \emph{iterated monodromy group} $\gal(K_{G,n}/K)$ is a subgroup of the automorphism group of a rooted tree. 

In a recent paper, Chonoles, Cullinan, Hausman, Pacelli, Pegado, and Wei \cite{cchppw14} define a new family of postcritically finite maps, the \emph{generalized Rikuna polynomials}, which also give rise to extensions with small Galois groups. These maps are named after a family of maps studied by Rikuna \cite{r02}, which are defined as follows.

Let $\ell>2$ be a positive integer, $K$ a field of characteristic coprime to $\ell$, and $\ring_K$ the ring of integers of $K$. Fix a primitive $\ell$-th root of unity $\zeta$ in a fixed algebraic closure $\overline K$ of $K$, and set $\zeta^+ = \zeta+\zeta\inv$. Assume further that $\zeta^+\in K$ but $\zeta \not\in K$. Define the polynomials $P(x;\ell)$ and $Q(x;\ell) \in K[x]$ by
\begin{align*}
P(x;\ell) = \frac{\zeta\inv(x-\zeta)^\ell - \zeta(x-\zeta\inv)^\ell}{\zeta\inv-\zeta}\quad \text{and} \quad Q(x;\ell) = \frac{(x-\zeta)^\ell - (x-\zeta\inv)^\ell}{\zeta\inv-\zeta}.
\end{align*}
Then the \emph{Rikuna polynomial} is
\begin{align*}
r(x,t;\ell) = P(x;\ell)-tQ(x;\ell) \in K(t)[x].
\end{align*}
This polynomial is \emph{generic} in that it parametrizes cyclic extensions of degree $\ell$ over $K$ that are not Kummer (since $\zeta\not\in K$). The case $\ell=3$ and $t=s/3$ yields Shanks' ``simplest cubic" polynomial.

Set $\varphi(x;\ell) =P(x;\ell)/Q(x;\ell)$, and define $P_n(x;\ell)$ and $Q_n(x;\ell)$ to be the polynomials that satisfy
\begin{align} \label{eq:phidef}
\varphi^n(x;\ell) = \frac{P_n(x;\ell)}{Q_n(x;\ell)},
\end{align}
where $\varphi^n$ denotes the $n$-fold composition of $\varphi$ (i.e., $\varphi^n = \varphi\circ\varphi^{n-1}$) expressed in lowest terms, and $P_n(x;\ell)$ is monic. Then the generalized Rikuna polynomial is
\begin{align} \label{eq:genrikdef}
r_n(x,t;\ell) = P_n(x;\ell)-tQ_n(x;\ell). 
\end{align}
Chonoles et al. note that when $K = \qq$, $\ell=3$, and $t\in \{0,1\}$, the tower of number fields coming from the iterates of the generalized Rikuna polynomial is ramified only at 3 (in fact, the tower is the cyclotomic $\zz_3$-extension of $\qq$), and they ask if there are any other specializations for which the number of ramified primes is small. In particular, it would be of great interest if the Rikuna polynomials could be used to produce extensions that are ramified only at $\ell$. Since these maps are postcritically finite, it is already known that any tower over $K$ will be finitely ramified. Indeed, using the discriminant formulas in \cite[Proposition 3.2 and Proposition 1, respectively]{ahm05,ch12}, one can show that the discriminant of $r_n(x,t;\ell)$ is
\begin{align} \label{eq:disc}
\disc r_n(x,t;\ell) &= \pm\ell^{n\ell^n}(\zeta-\zeta\inv)^{(\ell^n-2)(\ell^n-1)}(t^2-\zeta^+t+1)^{\ell^n-1}.
\end{align}
Moreover, there is a well known relationship between the $\ring_K$-ideals generated by $\disc r_n(x,t;\ell)$ and the relative discriminant of $K_n/K$:
\begin{align} \label{eq:disc relation}
(\disc r_n(x,t;\ell))\ring_K = [\ring_{K(\theta_n)} \colon \ring_K[\theta_n]]^2(\disc(K_n/K))\ring_K.
\end{align}
Thus one method for identifying the ramified primes in $K_n/K$ is to compute the ideal index $[\ring_{K(\theta_n)} \colon \ring_K[\theta_n]]$. The primary result of this paper is the computation of this index in the case $K = \qq$, $\ell=3$, and $t \in \zz$.

Let $v_p := \nu_p(t^2+t+1)$ denote the $p$-adic valuation of $t^2+t+1$, and let $\ind r_n(x,t;3) = [\ring_{\qq(\theta_n)} \colon \zz[\theta_n]]$.

\begin{thm} \label{th:maine index}
We have
\begin{align*}
\ind r_n(x,t;3) &= 3^{E/4}\prod_{p \mid t^2+t+1} p^{((3^n-1)(v_p-1) + \gcd(3^n,v_p)-1)/2}, \quad \text{where} \\
E &= 
\begin{cases*}
(3^n-1)^2+2V+2\sum_{k=0}^{V-1}3^{n-k} & if $t \equiv 1 \pmod 3$ \\
(3^n-1)(3^n-3) & otherwise,
\end{cases*}
\end{align*}
for a constant $V$ (depending on $n$) that we specify in Theorem \ref{th:ind3}. Moreover,
\begin{align*}
\disc K_n = 3^{n3^n-E/2}\prod_{p\mid t^2+t+1} p^{3^n-\gcd(3^n,v_p)}.
\end{align*}
\end{thm}

Immediately, we see that every prime dividing $t^2+t+1$ will also divide $\disc K_n$ once $n$ is sufficiently large. Moreover, $\ind r_n(x,t;3)>1$ whenever $n \ge 2$. This is in stark contrast to the discriminants of iterated extensions produced by Chebyshev polynomials (c.f. \cite[Theorem 1.1]{g13}).

The equation in Theorem \ref{th:maine index} is obtained via an application of the \emph{Montes algorithm}, which was developed by Gu\`ardia, Montes, and Nart in a recent series of papers \cite{gmn09,gmn11,gmn12}. Their algorithm detects the $p$-adic valuation of the index by gleaning arithmetic data from a refined version of the Newton polygon. We summarize their methods in Section \ref{sec:montes}, and in Section \ref{sec:maineproof} we prove Theorem \ref{th:maine index}. This algorithm, originally designed as a computational tool, has proven to be quite useful in computing the index associated to families of iterated maps. The complexity of the discriminant formula in Theorem \ref{th:maine index} is a testament to the power of the Montes algorithm, as it would be unlikely to recover this formula through observation alone. 

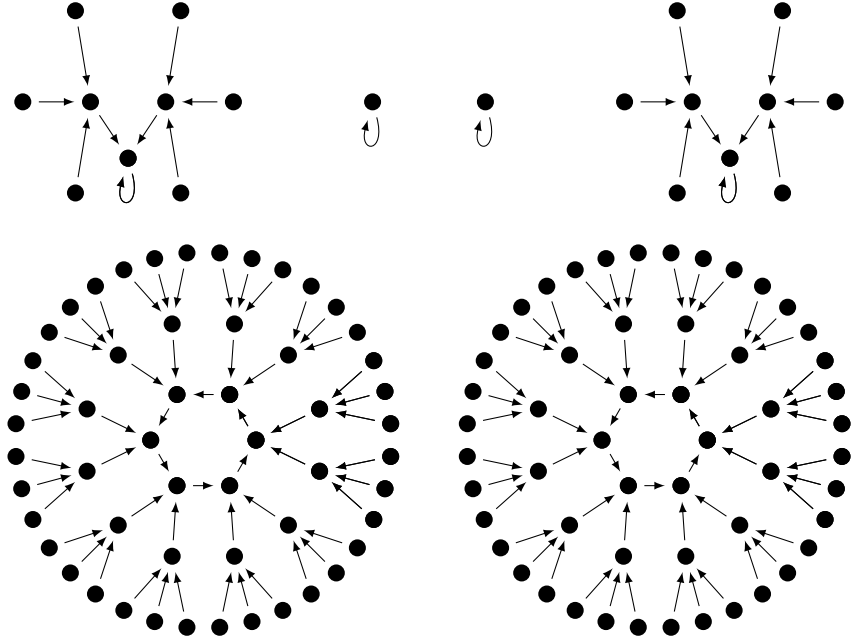
\begin{figure}[!ht]
\centering
\begin{itsapicture}
\begin{tikzpicture}[>=latex]
\tikzstyle{every node} = [fill,circle,outer sep = 1mm,inner sep = .8mm]
	\draw (-.75,0) node (a) {} edge[loop below,->] (a);
	\draw (.75,0) node (a) {} edge[loop below,->] (a);
	\draw (-3,-4.5) node[fill=none] (a) {};	
	\foreach \x in {0,...,6} {
		\draw[->] (a)++(\x*60:.7) node (1) {}
			(a)++(\x*60 + 60:.7) node (ahead)  {}
			(1) -- (ahead);
		\foreach \y in {-1,1} {
			\draw[->] (a)++(\x*60+\y*15:1.6) node (2) {}
				(2) -- (1);
			\foreach \z in {-1,0,1} {
				\draw[->] (a)++(\x*60+\y*15+\z*10:2.5) node (3) {}
					(3) -- (2);
			}
		}
	}
	\draw (3,-4.5) node[fill=none] (a) {};
	\foreach \x in {0,...,6} {
		\draw[->] (a)++(\x*60:.7) node (1) {}
			(a)++(\x*60 + 60:.7) node (ahead) {}
			(1) -- (ahead);
		\foreach \y in {-1,1} {
			\draw[->] (a)++(\x*60+\y*15:1.6) node (2) {}
				(2) -- (1);
			\foreach \z in {-1,0,1} {
				\draw[->] (a)++(\x*60+\y*15+\z*10:2.5) node (3) {}
					(3) -- (2);
			}
		}
	}
	\draw (-4,0) node[fill=none] (a) {};
	\foreach \x in {-1,1} {
		\draw[label distance = 2mm,->] (a)++(90+\x*90:.5) node (0) {}
			(a)+(-90:.75) node (1) {} 
			(0) -- (1)
			(1) edge[loop below] (1);
		\foreach \y in {-1,0,1} {
			\draw[->] (a)++(90+\x*90+\y*60:1.4) node (1) {}
				(1) -- (0);
		}
	}
	\draw (4,0) node[fill=none] (a) {};
	\foreach \x in {-1,1} {
		\draw[label distance = 2mm,->] (a)++(90+\x*90:.5) node (0) {}
			(a)+(-90:.75) node (1) {} 
			(0) -- (1)
			(1) edge[loop below] (1);
		\foreach \y in {-1,0,1} {
			\draw[->] (a)++(90+\x*90+\y*60:1.4) node (1) {}
				(1) -- (0);
		}
	}
\end{tikzpicture}
\end{itsapicture}
\caption{The graph of $\varphi(x;3)$ over $\pp\ff_{127}$.} \label{fig:g(127)}
\end{figure}

In Section \ref{sec:primes}, we consider the decomposition of primes in towers over $\qq(\zeta^+)$ generated by the generalized Rikuna polynomials. Our understanding comes from the study of certain graphs, a method proposed in \cite{ahm05}. Specifically, we study the dynamics of $\varphi(x;\ell)$ over $\pp\ff_q := \ff_q\cup\{\infty\}$, where $\ff_q$ is a finite field of characteristic $p$. The action of $\varphi(x;\ell)$ on $\pp\ff_q$ is naturally captured in a directed graph $G$: the vertices of $G$ correspond to the elements of $\pp\ff_q$, and $G$ contains a edge from $a$ to $b$ if and only if $\varphi(a;\ell) = b$. These graphs display an unusual degree of symmetry (see Figure \ref{fig:g(127)}) akin to that in graphs of the power maps and a few other families (see \cite{cs04,g14,u12,u13,vs04}). The following dynamical terms will be useful for describing these graphs. 

For any set $S$ and a map $f \colon S \to S$, we say that $a \in S$ is \emph{periodic} (for $f$) if $f^n(a) = a$ for some integer $n>0$; the smallest such integer being the period of $a$. If $f^m(a)$ is periodic for some $m \ge 0$, then $a$ is \emph{preperiodic}, and the smallest such integer is the preperiod of $a$. We use $\period(a)$ and $\pper(a)$ to denote the period and preperiod of $a$, respectively.

When $\ell$ is an odd prime, we give an explicit description of the decomposition of primes in iterated towers over $K=\qq(\zeta^+)$ in the following setting. Fix a tower $\{K_n\}$, and let $\fr p\subset \ring_K$ is a prime of norm $q$ satisfying $\fr p \nmid (\disc r_1(x,t;\ell))\ring_K$ and $q \equiv 1 \pmod \ell$. Consider the graph of $\varphi(x;\ell)$ over $\pp\ff_q$, and set 
\begin{align*}
M = \max_{a \in \pp\ff_q}\{\pper(a)\}
\end{align*}
Let $\pper(t)$ denote the preperiod of $(t \bmod{\fr p}) \in \ff_q$. We will prove the following.

\begin{thm} \label{th:maine decomposition}
If $\pper(t) \ge 1$, then $\fr p$ splits completely in $K_n$ for each $1 \le n \le M - \pper(t)$ and is completely inert afterwards. Otherwise, $\fr p$ splits completely in $K_n$ for each $1 \le n \le M$, and for each $n > M$, $K_n$ contains $\sum_{i=1}^M \ell^i$ primes of degree 1 and $\ell^M$ primes of degree $\ell^i$ for each $1 \le i \le n-M$ over $\fr p$.
\end{thm}

\begin{remark}
Prior to Rikuna, Shen and Washington \cite{sw94,sw95} studied an overlapping family of polynomials,  $r_n(x,a/p^n;\ell)$ (written in our notation). The polynomials $P_n(x;\ell)$ and $Q_n(x;\ell)$ defined in Equation \eqref{eq:phidef} coincide with the polynomials $R_n(x)$ and $S_n(x)$, respectively, defined in \cite{sw95}.
\end{remark}

\section{Preliminaries}
Let $\ell >2$ be an arbitrary integer, and let $\zeta$ be a primitive $\ell$-th root of unity. For the remainder of the paper, we fix $K = \qq(\zeta^+)$. We derive two expressions for $r_n(x,t;\ell)$ that do not appear in \cite{cchppw14,sw95}. The first is a generalization of \cite[Corollary 2.6]{r02}. These formulations yield simple factorization results modulo certain ideals.

\begin{prop} \label{prop:genrik}
The generalized Rikuna polynomial satisfies
\begin{align*}
r_n(x,t;\ell) &= \frac{(t-\zeta)(x-\zeta\inv)^{\ell^n} - (t-\zeta\inv)(x-\zeta)^{\ell^n}}{\zeta\inv-\zeta}.\quad  
\end{align*} 
\end{prop}

\begin{proof}
From Equation \eqref{eq:phidef}, we have
\begin{align*}
\varphi(x;\ell)-\zeta =  \frac{(\zeta\inv-\zeta)(x-\zeta)^\ell}{(x-\zeta)^\ell-(x-\zeta\inv)^\ell}, \quad \text{ and } \quad \varphi(x;\ell) - \zeta\inv = \frac{(\zeta\inv-\zeta)(x-\zeta\inv)^\ell}{(x-\zeta)^\ell-(x-\zeta\inv)^\ell}.
\end{align*}
By induction on $n$, 
\begin{align*}
\varphi^n(x;\ell) &= \frac{\zeta\inv(x-\zeta)^{\ell^n} - \zeta(x-\zeta\inv)^{\ell^n}}{(x-\zeta)^{\ell^n} - (x-\zeta\inv)^{\ell^n}}.
\end{align*}
We introduce a normalizing factor so that the numerator is monic, we obtain
\begin{align} \label{eq:pq}
P_n(x;\ell) &= \frac{\zeta\inv(x-\zeta)^{\ell^n} - \zeta(x-\zeta\inv)^{\ell^n}}{\zeta\inv-\zeta} \quad \text{and} \quad
Q_n(x;\ell) = \frac{(x-\zeta)^{\ell^n} - (x-\zeta\inv)^{\ell^n}}{\zeta\inv-\zeta},
\end{align}
from which the result follows.
\end{proof}

The generalized Rikuna polynomial can also be expressed in terms of Chebyshev polynomials of the second kind. The degree $n$ Chebyshev polynomial of the second kind, denoted by $U_n$, is the unique polynomial that satisfies the trigonometric relation
\begin{align} \label{eq:chtrig}
U_n(2\cos\theta) = \frac{\sin((n+1)\theta)}{\sin(\theta)}.
\end{align}

\begin{prop} \label{prop:genrikch}
The generalized Rikuna polynomial is given by
\begin{align*}
r_n(x,t;\ell) &= x^{\ell^n}-\ell^ntx^{\ell^n-1} + \sum_{k=2}^{\ell^n}(-1)^k{\ell^n\choose k}\Big(t\,U_{k-1}(\zeta^+)-U_{k-2}(\zeta^+)\Big)x^{\ell^n-k}.
\end{align*}
\end{prop}

\begin{proof}
Using binomial expansion and Equation \eqref{eq:chtrig}, we obtain
\begin{align} \label{eq:altpq}
P_n(x;\ell) &= \frac{\zeta\inv(x-\zeta)^{\ell^n} - \zeta(x-\zeta\inv)^{\ell^n}}{\zeta\inv-\zeta} = \sum_{k=0}^{\ell^n}(-1)^k{\ell^n\choose k}\frac{\zeta^{k-1}-\zeta^{-(k-1)}}{\zeta\inv-\zeta}x^{\ell^n-k}  \\
&= x^{\ell^n}-\sum_{k=2}^{\ell^n}(-1)^k{\ell^n\choose k} U_{k-2}(\zeta^+)x^{\ell^n-k}, \quad \text{and} \nonumber\\
Q_n(x;\ell) &= \frac{(x-\zeta)^{\ell^n} - (x-\zeta\inv)^{\ell^n}}{\zeta\inv-\zeta} = \sum_{k=0}^{\ell^n} (-1)^k {\ell^n\choose k} \frac{\zeta^k - \zeta^{-k}}{\zeta\inv-\zeta}x^{\ell^n-k} \nonumber\\
&= \ell^nx^{\ell^n-1}-\sum_{k=2}^{\ell^n}(-1)^k{\ell^n\choose k} U_{k-1}(\zeta^+)x^{\ell^n-k}.\nonumber
\end{align}
The result follows from Equation \eqref{eq:genrikdef}.
\end{proof}

We now give two simple factorization results that will be useful in our later analysis.

\begin{lem} \label{lem:rmodp}
Let $\fr a \subset \ring_K$ be any ideal for which $t = \zeta \bmod{\fr a}$. Then $r_n(x,t;\ell) = (x-t)^{\ell^n} \in (\ring_K/\fr a)[x]$.
\end{lem}

\begin{proof}
The result follows immediately from Proposition \ref{prop:genrik}.
\end{proof}

Note that if $\ell$ is an odd prime, then the ideal $(\ell)$ is totally ramified in $K = \qq(\zeta^+)$. 

\begin{lem} \label{lem:rmodl}
Let $\ell$ be an odd prime, and let $\fr l \subset \ring_K$ be the prime ideal containing $\ell$. Then $r_n(x,t;\ell) = (x-1)^{\ell^n} \in (\ring_K/\fr l)[x]$.
\end{lem}

\begin{proof}
Note that $\zeta^+ = 2\cos(2n/\ell)$, so by Equation \eqref{eq:chtrig},
\begin{align*}
U_{\ell^n-1}(\zeta^+) = \frac{\sin(2n\ell^{n-1})}{\sin(2n/\ell)} = 0, \quad \text{and} \quad U_{\ell^n-2}(\zeta^+) = \frac{\sin(2n\ell^{n-1}-2n/\ell)}{\sin(2n/\ell)} = -1.
\end{align*}
Combined with Proposition \ref{prop:genrikch}, we have
\begin{align*}
r_n(x,t;\ell) = x^{\ell^n}-t\,U_{\ell^n-1}(\zeta^+) + U_{\ell^n-2}(\zeta^+) = x^{\ell^n}-1 = (x-1)^{\ell^n} \in (\ring_K/\fr l)[x].
\end{align*}
\end{proof}

\section{Montes algorithm} \label{sec:montes}

In this section, we describe the Montes algorithm \cite{gmn09,gmn11,gmn12}, which is the main tool we use to compute the expression for the index in Theorem \ref{th:maine index}. The key result is Theorem \ref{th:gmn}, which relates the $p$-adic valuation of the index to specialized Newton polygons. We proceed to set up the notation for understanding this theorem following the presentations in \cite{fmn09,g13}.

For a general number field $\qq(\theta)$, where $\theta$ is an algebraic integer with minimal polynomial $f$, let
\begin{align*}
\ind f := [\ring_{\qq(\theta)} \colon \zz[\theta]], \quad \text{and} \quad \ind_p f := \nu_p(\ind f),
\end{align*}
where $\nu_p$ denotes the standard $p$-adic valuation.

The Montes algorithm computes $\ind_p f$ as follows. First, factor $f$ modulo $p$ into irreducibles:
\begin{align*}
f(x) \equiv \phi_1(x)^{e_1}\cdots \phi_s(x)^{e_s}\pmod p.
\end{align*}
For each factor $\phi :=\phi_i$, the \emph{$\phi$-development of} $f$ is the expansion
\begin{align*}
f(x) = a_0(x)\phi(x) + a_1(x)\phi(x)^2 + \cdots + a_d(x)\phi(x)^d.
\end{align*}
Let $\eucal N_\phi$ denote the lower convex hull of the set of points $\left\{ \big(k,\nu_p(a_k(x))\big) \colon 0 \le k \le d \right\}$, where the $p$-adic valuation of a polynomial $a(x) = \sum b_i x^i$ is defined to be 
\begin{align*}
\nu_p(a(x)) = \min \{\nu_p(b_i)\}.
\end{align*}
The edges of $\eucal N_\phi$ with negative slope form the \emph{$\phi$-Newton polygon}, denoted by $\eucal N_\phi^-$. 

\begin{figure}[!ht]
\centering
\begin{tikzpicture}[scale=.6]
\foreach \x in {0,8} {
	\draw[gray] (\x-.5,0)--(\x+6.5,0) (\x+0,-.5)--(\x+0,3.5);
	\foreach \y in {1,2,3}
		\draw (.1+\x,\y) -- (-.1+\x,\y) node[left,font=\footnotesize]{\y};
	\foreach \y in {1,...,6}
		\draw (\y+\x,.1) -- (\y+\x,-.1) node[below,font=\footnotesize]{\y};
	\foreach \y/\z in {0/3,2/1,3/2,4/0,5/0,6/2}
		\fill (\x+\y,\z) circle (3pt);
	\draw[line width = .5mm] (\x+0,3) -- (\x+2,1) -- (\x+4,0);
	}
\draw[line width=.5mm] (4,0)--(5,0)--(6,2);
\node at (3,3) {$\eucal N_\phi$};
\node at (11,3) {$\eucal N_\phi^-$};
\end{tikzpicture}
\caption{A $\phi$-Newton polygon (left) and its principal part (right).}\label{fig:principal}
\end{figure}
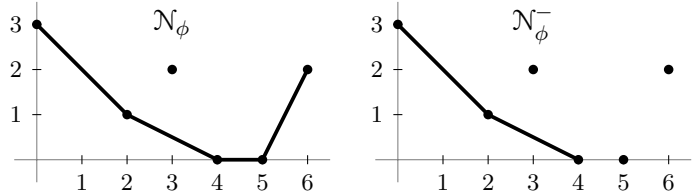

We also define a reduction map 
\begin{align*}
\red_\phi \colon \zz[x] &\to \ff_p[x]/(\phi) \\
a(x) &\mapsto \overline{a(x)/p^{\nu_p(a(x))}}.
\end{align*}
To each lattice point on the $\phi$-Newton polygon, we attach a residual coefficient
\begin{align*}
\res(k) = 
\begin{cases*}
\red_\phi(a_k(x)) & if $\big(k,\nu_p(a_k(x))\big)$ is on $\eucal N_\phi^-$\\
0 & if $\big(k,\nu_p(a_k(x)\big)$ is above $\eucal N_\phi^-$.
\end{cases*}
\end{align*}
Let $S$ be an side of $\eucal N_\phi^-$, and denote the left and right endpoints of $S$ by $(x_0,y_0)$ and $(x_1,y_1)$, respectively. The \emph{degree} of $S$, denoted by $\deg S$, is $\gcd(y_1-y_0,x_1-x_0)$. The \emph{residual polynomial} associated to $S$ is
\begin{align*}
R_S(y) = \sum_{i=0}^{\deg S} \res\left(x_0+i\frac{(x_1-x_0)}{\deg S}\right)y^i.
\end{align*}
We note that $\res(x_0)$ and $\res(x_1)$ are necessarily non-zero, and in particular, it is always the case that $\deg S = \deg R_S$.

\begin{example} \label{ex:3.1}
Consider the irreducible polynomial $f(x) = x^4 + 23x^3 + 12x^2 + 11x + 7$, which factors over $\ff_3[x]$ into $f(x) \equiv (x+2)^4 \pmod 3$. Set $\phi(x) = x+2$, then the $\phi$-development of $f$ is
\begin{align*}
f(x) =  - 135 + 207(x+2) - 102(x+2)^2 + 15(x+2)^3 + (x+2)^4.
\end{align*}
The $\phi$-Newton polygon has two sides: one of slope $-1$ and degree 2, the other of slope $-1/2$ and degree 1. The residual coefficients are $c_0 = 1$, $c_1 = -1$, $c_2 = -1$, $c_3 = 0$, and $c_4 = 1$, and the residual polynomials attached to these sides are $R_1(f)(y) = -y^2+1$ and $R_2(f)(y) = y-1$, respectively. See Figure \ref{fig:polygon}.

\begin{figure}[!ht]
\begin{tikzpicture}
\draw[gray] (-.5,0)--(4.5,0) (0,-.5) -- (0,3.5);
\foreach \x in {1,2,3}
	\draw (.1,\x) -- (-.1,\x) node[left,font=\footnotesize]{\x};
\foreach \x in {1,2,3,4}
	\draw (\x,.1) -- (\x,-.1) node[below,font=\footnotesize]{\x};
\foreach \x/\y in {0/3,1/2,2/1,3/1,4/0}
	\fill (\x,\y) circle (3pt);
\draw[line width = .5mm] (0,3) -- node[above right] {$S_1$} (2,1) -- node[right=3mm] {$S_2$}(4,0);
\node at (1,1) {+};
\end{tikzpicture}
\caption{The $(x+2)$-Newton polygon for $f(x) = x^4 + 23x^3 + 12x^2 + 11x + 7$ at $p=3$.} \label{fig:polygon}
\end{figure}
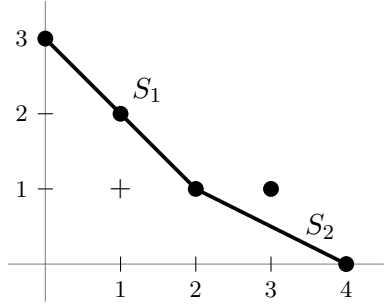
\end{example}

We say that $f$ is \emph{$\phi$-regular} if $R_S(y)$ is separable for each side $S$ of $\eucal N^-_\phi$, and $f$ is \emph{$p$-regular} if $f$ is $\phi$-regular for each irreducible factor $\phi$ of $f$. Finally, let $L_\phi$ denote the number of integral lattice points $(x,y)$ strictly in the first quadrant (that is, $x>0$ and $y>0$) on or under $\eucal N_\phi^-$, and define
\begin{align*}
\ind_\phi f = L_\phi\cdot\deg\phi.
\end{align*}

\begin{thm} \label{th:gmn}
Let $f \in \zz[x]$ be a monic and separable polynomial. Then 
\begin{align*}
\ind_p f \ge \sum_{i=0}^s \ind_{\phi_i} f
\end{align*}
with equality if and only if $f$ is $p$-regular.
\end{thm}

\begin{proof}
See \cite[Section 4]{gmn12}.
\end{proof}

\begin{example}[Example \ref{ex:3.1} continued]
Returning to the previous example, we see that both of the residual polynomials are separable over $\ff_3[y]$, hence $f$ is 3-regular. By Theorem \ref{th:gmn}, we conclude that $\ind_3(f) = 3$ since $\deg \phi = 1$ and there are three points with integral coordinates on or below the polygon. This result is verified in \texttt{PARI} \cite{PARI}.
\end{example}

Pick's Theorem \cite{p99} provides a simple method for counting lattice points inside bounded regions of the plane.

\begin{lem}[Pick's Theorem] \label{lem:pick}
Let $A$ be the area of a simple closed lattice polygon. Let $B$ denote the number of lattice points on the Polygon sides and $I$ the number of lattice points in the interior of the polygon. Then $I = A +1- B/2$.
\end{lem}

\section{Shanks' specialization: $\ell=3$} \label{sec:maineproof}
In this section, we apply the Montes algorithm to compute the index associated to the generalized Rikuna polynomials specialized at $\ell=3$, proving Theorem \ref{th:maine index}. Unless otherwise noted, we set $r_n(x,t):= r_n(x,t;3)$ throughout this section. From the discriminant formula in Equation \eqref{eq:disc}, we know that the only primes that could divide the index are 3 and the primes dividing $t^2+t+1$. We address these cases separately.

\subsection{Index calculation: $p=3$}
Let $u_n:=U_n(1)$. From Proposition \ref{prop:genrikch},
\begin{align} \label{eq:rik3}
r_n(x,t) &= \sum_{k=0}^{3^n} {3^n\choose k} \Big(tu_{k+2}+u_{k+1}\Big)x^{3^n-k} = \sum_{k=0}^{3^n}{3^n\choose k}\Big(tu_{k+2}-u_k\Big)x^k.
\end{align}
Recall from Lemma \ref{lem:rmodl} that $r_n(x,t) \equiv (x-1)^{3^n}\pmod 3$, hence there is only one factor, $\phi(x) = x-1$, to consider. We obtain the $(x-1)$-development of $r_n(x,t)$ using Taylor expansion:
\begin{align} \label{eq:taylor3}
r_n(x,t) = \sum_{m=0}^{3^n}\frac{r^{(m)}_n(1,t)}{m!}(x-1)^m,
\end{align}
where $r^{(m)}_n(x,t)$ denotes the $m$-th derivative of $r_n(x,t)$ with respect to $x$. By Equation \eqref{eq:rik3}, we have
\begin{align*}
r_n^{(m)}(x,t)&=\sum_{k=0}^{3^n-m}{3^n\choose k+m}\frac{(k+m)!}{k!}\Big(tu_{k+m+2}-u_{k+m}\Big)x^k \\
&= \frac{3^n!}{(3^n-m)!}\sum_{k=0}^{3^n-m}{3^n-m\choose k} \Big(tu_{k+m+2}-u_{k+m}\Big)x^k.
\end{align*}
The coefficients in this expression are given by two known sequences (see \cite{oeis} \href{http://oeis.org/A057681}{A057681}, \href{http://oeis.org/A057083}{A057083}):
\begin{align*}
&\sum_{k=0}^{3^n-m}{3^n-m\choose k} u_{k+m+2} = \sum_{k=0}^{\floor{\frac{3^n-m-1}{3}}}(-1)^{k+1}{3^n-m+1 \choose 3k+2} = (-27)^{\floor{\frac{3^n-m}{6}}}b_m, \quad \text{and}\\
-&\sum_{k=0}^{3^n-m}{3^n-m\choose k} u_{k+m} = \sum_{k=0}^{\floor{\frac{3^n-m+1}{3}}}(-1)^k{3^n-m+1\choose 3k} = (-27)^{\floor{\frac{3^n-m}{6}}}c_m,
\end{align*}
where
\begin{align*}
b_m = 
\begin{cases*}
-6 & if $m \equiv 0 \pmod 6$\\
-3 & if $m  \equiv 1 \pmod 6$\\
-1 & if $m \equiv 2 \pmod 6$\\
0 & if $m \equiv 3 \pmod 6$\\
-9 & if $m \equiv 4 \pmod 6$\\
-9 & if $m \equiv 5 \pmod 6$
\end{cases*}
\qquad \text{and} \qquad
c_m = \begin{cases*}
-3 & if $m \equiv 0 \pmod 6$\\
0 & if $m \equiv 1 \pmod 6$\\
1 & if $m  \equiv 2 \pmod 6$\\
1 & if $m \equiv 3 \pmod 6$\\
-18 & if $m \equiv 4 \pmod 6$\\
-9 & if $m \equiv 5 \pmod 6$.
\end{cases*}
\end{align*}
Thus setting
\begin{align} \label{eq:a_n,m}
&a_{n,m} = {3^n\choose m}\sum_{k=0}^{3^n-m}{3^n-m\choose k} \Big(tu_{k+m+2}-u_{k+m}\Big),\nonumber\\
\intertext{we have}
&a_{n,m}=\pm{3^n\choose m}3^{\floor{(3^n-m)/2}}e_m(t), \quad \text{where}\\
&e_m(t)=
\begin{cases*}
2t+1 & if $m \equiv 0 \pmod 6$ \\
t & if $m \equiv 1 \pmod 6$ \\
t-1 & if $m \equiv 2 \pmod 6$ \\
1 & if $m \equiv 3 \pmod 6$ \\
t+2 & if $m \equiv 4 \pmod 6$ \\
t+1 & if $m \equiv 5 \pmod 6$.
\end{cases*}\nonumber
\end{align}
Equation \eqref{eq:taylor3} is now reduced to 
\begin{align*}
r_n(x,t) = \sum_{m=0}^{3^n}\frac{r_n^{(m)}(1,t)}{m!}(x-1)^m = \sum_{m=0}^{3^n}a_{n,m}(x-1)^m,
\end{align*}
and we are left to compute the $3$-adic valuation of $a_{n,m}$.

\begin{thm} \label{th:ind3a}
If $t \not \equiv 1 \pmod 3$, then $\ind_3 r_n(x,t) = (3^n-1)(3^n-3)/4$.
\end{thm}

\begin{proof}
By Equation \eqref{eq:a_n,m}, we have
\begin{align*}
\nu_3(a_{n,0}) = \frac{3^n-1}{2}, \qquad \nu_3(a_{n,3^n}) = 0,
\end{align*}
and for each $m$ satisfying $0 < m < 3^n$,
\begin{align} \label{eq:ind3a valuation}
\nu_3(a_{n,m}) &= \nu_3{3^n \choose m}+\floor{\frac{3^n-m}{2}}+ \nu_3(e_m(t)) \ge 1+\floor{\frac{3^n-m}{2}} \ge \frac{3^n-m+1}{2}.
\end{align}
One may now verify that each vertex $\left(m,\nu_3( a_{n,m})\right)$ lies above the line segment joining $\left(0,\frac{3^n-1}{2}\right)$ to $\left(3^n,0\right)$, hence the $(x-1)$-Newton polygon is comprised of one side. See Figure \ref{fig:ind3a figure}. Moreover, the degree of the residual polynomial associated to this side is 1, hence by Theorem \ref{th:gmn}, the 3-valuation of the index is equal to the number of lattice points contained inside the triangle with vertices $\left(0,\frac{3^n-1}{2}\right)$, $\left(3^n,0\right)$, and $(0,0)$. Applying Lemma \ref{lem:pick}, we obtain the result.
\end{proof}

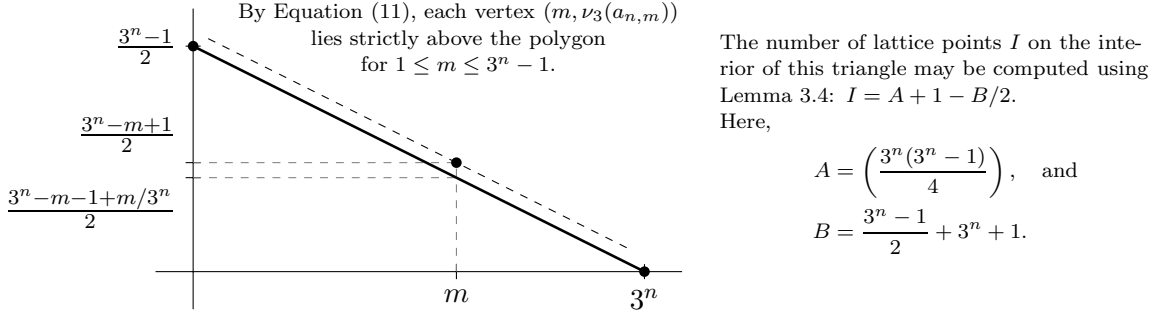
\begin{figure}[!ht]
\centering
\begin{tikzpicture}
\draw (-.5,0) -- (6.5,0) 
	(0,-.5) -- (0,3.5)
	(-.1,3) node[left]{$\frac{3^n-1}{2}$}-- (.1,3) 
	(6,-.1) node[below]{$3^n$} -- (6,.1) 
	(3.5,-.1) node[below]{$m$} -- (3.5,.1)
	(-.1,1.25) node[below left] {$\frac{3^n-m-1+m/3^n}{2}$} -- (.1,1.25)
	(-.1,1.45) node[above left] {$\frac{3^n-m+1}{2}$} -- (.1,1.45);
\draw[dashed,gray] (0,1.25) -- (3.5,1.25) (0,1.45) -- (3.5,1.45) (3.5,0) -- (3.5,1.45);
\draw[dashed] (.2,3.1) node[right,font=\scriptsize,text width=6 cm,text badly centered,outer sep = 2mm] {By Equation \eqref{eq:ind3a valuation}, each vertex $(m,\nu_3(a_{n,m}))$ \par lies strictly above the polygon \par for $1 \le m \le 3^n-1$.} -- (5.8,.3);
\fill (3.5,1.45) circle (2pt)
	(0,3) circle (2pt)
	(6,0) circle (2pt);
\draw[line width=1pt] (0,3) -- (6,0);
\node[text width=6cm,font=\scriptsize] at (10,1.7) {The number of lattice points $I$ on the interior of this triangle may be computed using Lemma \ref{lem:pick}: $I = A+1-B/2$.\par Here, \begin{align*}A &= \left(\frac{3^n(3^n-1)}{4}\right), \quad \text{and} \\ B &= \frac{3^n-1}{2}+3^n+1.\end{align*}};
\end{tikzpicture}
\caption{When $t\equiv 1 \pmod 3$, the $(x-1)$-Newton polygon is one sided.}\label{fig:ind3a figure}
\end{figure}

Note that for even values of $m$, the 3-adic valuation of $e_m(t)$ (and hence $a_{n,m}$) can be made arbitrarily large by taking appropriate values of $t$ congruent to 1 modulo 3. For the same values of $t$, the 3-adic valuation of $a_{n,m}$, where $m$ is odd, remains unchanged. Hence, as $t$ varies, the $(x-1)$-Newton polygon is dictated by the vertices with odd abscissae. In fact, it is enough to consider the abscissae that are powers of 3.

\begin{prop} \label{prop:x-1polygon}
If $t \equiv 1 \pmod 3$, then the $(x-1)$-Newton polygon is the lower convex hull of the set of points 
\begin{align*}
\left\{\Big(0,\nu_3(a_{n,0})\Big)\right\} \cup \left\{\left(3^k,n-k+\frac{3^n-3^k}{2}\right) \colon 1 \le k \le n\right\}.
\end{align*}
\end{prop}

\begin{proof}
It is known that $\nu_3{3^n\choose m} = n-\nu_3(m)$ for each $0<m<3^n$, and from Equation \eqref{eq:a_n,m} we have
\begin{align} \label{eq:x-1polygon}
\nu_3(a_{n,m}) &= n-\nu_3(m)+\frac{3^n-m}{2} \quad \text{if $m$ is odd, and}\\
\nu_3(a_{n,m}) &\ge n-\nu_3(m)+\frac{3^n-m}{2}+1 \quad \text{if $m$ is even.}\nonumber
\end{align}
For each $m$ satisfying $3^k < m< 3^{k+1}$, it is a straightforward calculation to verify that the vertex $(m,\nu_3(a_{n,m}))$ lies strictly above the line segment joining $(3^k,\nu_3(a_{n,3^k}))$ and $(3^{k+1},\nu_3(a_{n,3^{k+1}}))$. See Figure \ref{fig:x-1polygon figure}. 
\end{proof}

\begin{figure}[!ht]
\centering
\begin{tikzpicture}
\draw (-.5,0) -- (4.5,0)
	(0,-.1) node[below]{$3^k$}--(0,.1)
	(4,-.1) node[below]{$3^{k+1}$} -- (4,.1)
	(2.5,-.1) node[below]{$m$} -- (2.5,.1);
\draw[dashed] (2.5,0) -- (2.5,2.5);
\draw[line width=1pt] (0,3) -- node[above,near start] {$S$} (4,1);
\fill (0,3) circle (2pt) node[left] {$\left(3^k,n-k+\frac{3^n-3^k}{2}\right)$}
	(4,1) circle (2pt) node[right] {$\left(3^{k+1},n-k-1+\frac{3^n-3^{k+1}}{2}\right)$} 
	(2.5,1.75) circle (2pt) node[below left] {$\left(m,n-k+\frac{3^n-m}{2}-\frac{m-3^k}{3^{k+1}-3^k}\right)$}
	(2.5,2.5) circle (2pt) node[above right,font=\scriptsize,text width=6.5cm] {By Equation \eqref{eq:x-1polygon}, $\nu_3(a_{n,m}) \ge n-k+\frac{3^n-m}{2}+1$ for each $3^k<m<3^{k+1}$.};
\end{tikzpicture}
\caption{Each vertex n the interval $(3^k,3^{k+1})$ lies above the side $S$.} \label{fig:x-1polygon figure}
\end{figure}
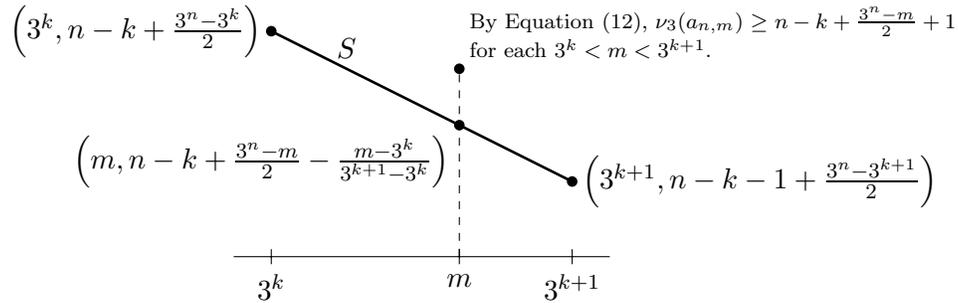

\begin{example}
The possible $(x-1)$-Newton polygons for $r_2(x,t)$ when $t \equiv 1 \pmod 3$ are shown in Figure \ref{fig:x-1 example}. By Proposition \ref{prop:x-1polygon}, the polygon is the lower convex hull of the points $\{(0,\nu_3(a_{2,0})),$ $(1,6), (3,4), (9,0)\}$. By Equation \eqref{eq:a_n,m}, $\nu_3(a_{2,0}) \ge 5$. The lower convex hull of the points $\{ (1,6), (3,4),$ $(9,0)\}$ is marked by a dashed line.
\end{example}

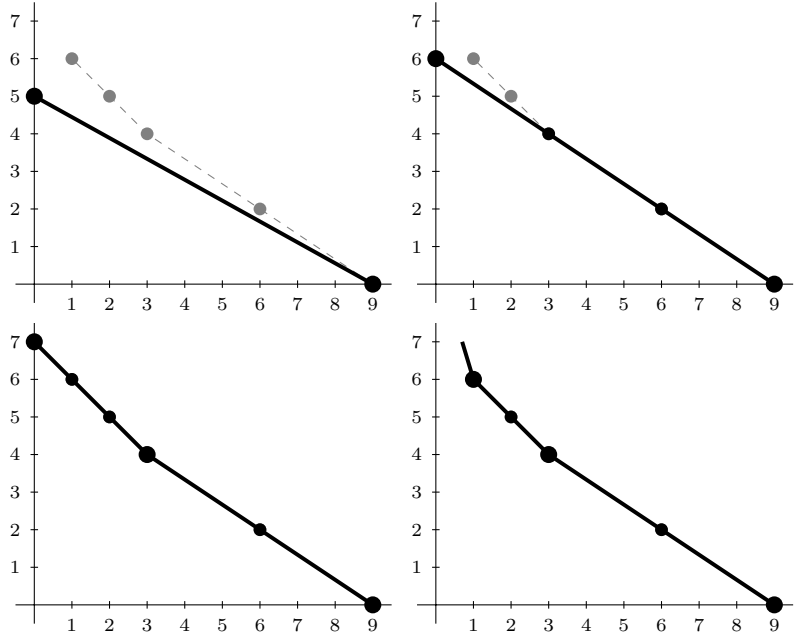
\begin{figure}[!ht]
\centering
\begin{tikzpicture}[scale=.5]
\draw (-.5,0) -- (9.5,0) 
	(0,-.5) -- (0,7.5);
\draw[gray,dashed] (9,0) -- (3,4) -- (1,6);
\foreach \x in {1,...,9}
	\draw (\x,-.1) node[below,font=\tiny] {$\x$} -- (\x,.1);
\foreach \x in {1,...,7}
	\draw (-.1,\x) node[left,font=\tiny] {$\x$} -- (.1,\x);
\foreach \x/\y in {1/6,2/5,3/4,6/2,9/0}
	\draw (\x,\y) node[gray,fill,circle,inner sep=.6mm] {};
\draw[line width=1.5pt] (9,0) -- (0,5);
\foreach \x/\y in {9/0,0/5}
	\draw (\x,\y) node[fill,circle,inner sep=.8mm] {};
\end{tikzpicture}
\begin{tikzpicture}[scale=.5]
\draw (-.5,0) -- (9.5,0) 
	(0,-.5) -- (0,7.5);
\draw[gray,dashed] (9,0) -- (3,4) -- (1,6);
\foreach \x in {1,...,9}
	\draw (\x,-.1) node[below,font=\tiny] {$\x$} -- (\x,.1);
\foreach \x in {1,...,7}
	\draw (-.1,\x) node[left,font=\tiny] {$\x$} -- (.1,\x);
\foreach \x/\y in {1/6,2/5,3/4,6/2,9/0}
	\draw (\x,\y) node[gray,fill,circle,inner sep=.6mm] {};
\draw[line width=1.5pt] (9,0) -- (0,6);
\foreach \x/\y in {9/0,0/6}
	\draw (\x,\y) node[fill,circle,inner sep=.8mm] {};
\foreach \x/\y in {3/4,6/2}
	\draw (\x,\y) node[fill,circle,inner sep=.6mm] {};
\end{tikzpicture}

\begin{tikzpicture}[scale=.5]
\draw (-.5,0) -- (9.5,0) 
	(0,-.5) -- (0,7.5);
\draw[gray,dashed] (9,0) -- (3,4) -- (1,6);
\foreach \x in {1,...,9}
	\draw (\x,-.1) node[below,font=\tiny] {$\x$} -- (\x,.1);
\foreach \x in {1,...,7}
	\draw (-.1,\x) node[left,font=\tiny] {$\x$} -- (.1,\x);
\foreach \x/\y in {1/6,2/5,3/4,6/2,9/0}
	\draw (\x,\y) node[gray,fill,circle,inner sep=.6mm] {};
\draw[line width=1.5pt] (9,0) -- (3,4) -- (0,7);
\foreach \x/\y in {9/0,3/4,0/7}
	\draw (\x,\y) node[fill,circle,inner sep=.8mm] {};
\foreach \x/\y in {6/2,2/5,1/6}
	\draw (\x,\y) node[fill,circle,inner sep=.6mm] {};
\end{tikzpicture}
\begin{tikzpicture}[scale=.5]
\draw (-.5,0) -- (9.5,0) 
	(0,-.5) -- (0,7.5);
\draw[gray,dashed] (9,0) -- (3,4) -- (1,6);
\foreach \x in {1,...,9}
	\draw (\x,-.1) node[below,font=\tiny] {$\x$} -- (\x,.1);
\foreach \x in {1,...,7}
	\draw (-.1,\x) node[left,font=\tiny] {$\x$} -- (.1,\x);
\foreach \x/\y in {1/6,2/5,3/4,6/2,9/0}
	\draw (\x,\y) node[gray,fill,circle,inner sep=.6mm] {};
\draw[line width=1.5pt] (9,0) -- (3,4) -- (1,6) -- (.7,7);
\foreach \x/\y in {9/0,3/4,1/6}
	\draw (\x,\y) node[fill,circle,inner sep=.8mm] {};
\foreach \x/\y in {6/2,2/5}
	\draw (\x,\y) node[fill,circle,inner sep=.6mm] {};
\end{tikzpicture}
\caption{Possible $(x-1)$-Newton polygons for $r_2(x,t)$ when $t \equiv 1 \pmod 3$.} \label{fig:x-1 example}
\end{figure}

\begin{thm} \label{th:ind3}
Suppose $t \equiv 1 \pmod 3$. Set $V = \min\{\nu_3(a_{n,0}) - \frac{3^n+1}{2},n\}$. Then
\begin{align*}
\ind_3 r_n(x,t) = \frac{1}{4}\left((3^n-1)^2+2V+2\sum_{k=0}^{V-1}3^{n-k}\right).
\end{align*}
\end{thm}

\begin{proof}
Since $t \equiv 1 \pmod 3$, we have $\nu_3(a_{n,0}) \ge \frac{3^n+1}{2}$ by Equation \eqref{eq:a_n,m}. There are three possibilities for the shape of the polygon, depending on the $3$-adic valuation of $a_{n,0}$.

If $\nu_3(a_{n,0})= \frac{3^n+1}{2}$, then by Proposition \ref{prop:x-1polygon}, the $(x-1)$-Newton polygon has a single side joining the vertices $\left( 0 , \frac{3^n+1}{2}\right)$ and $(3^n,0)$. The degree of the associated residual polynomial is 1.

If $\frac{3^n+1}{2} < \nu_3(a_{n,0}) \le \frac{3^n+1}{2}+n$, then the polygon has $V$ sides. The leftmost side, which lies above the interval $[0,3^{n-V+1}]$, is degree 3, as it passes through lattice points at $x=3^{n-V}$ and $x=2\cdot 3^{n-V}$. From the proof of Proposition \ref{prop:x-1polygon}, the point $\left(2\cdot3^{n-V},\nu_3(a_{n,2\cdot 3^{n-V}})\right)$ lies above the polygon. Hence the residual coefficient attached to this point is 0, and the residual polynomial associated to this side is of the form $y^3\pm y \pm 1$. The degree of each subsequent side is 2, and by the same analysis, the residual polynomials associated to these sides are of the form $y^2\pm1$.

If $\nu_3(a_{n,0})>\frac{3^n+1}{2}$, then the polynomial is comprised of $V+1$ sides. The leftmost side lies above the interval $[0,1]$ and is degree 1. All other sides are degree 2 and have residual polynomials of the form $y^2\pm1$.

In all three cases, the residual polynomials are separable over $\ff_3[x]$, hence $r_n(x,t)$ is 3-regular. By Theorem \ref{th:gmn} and Lemma \ref{lem:pick}, we have
\begin{align} \label{eq:ind3}
\ind_3 r_n(x,t) = A+1-B/2 + 2V
\end{align}
where $A$ is the area of the region in the first quadrant below the polygon, $B$ is the number of lattice points on the boundary of this region, and $2V$ is the number of lattice points on the $(x-1)$-Newton polygon. We compute $A$ by breaking the region into triangles with their bases along the $y$-axis to obtain
\begin{align*}
& A = \frac{3^n(3^n+1)}{4} + \frac{1}{2}\sum_{k=0}^{V-1}3^{n-k} \quad \text{and} \quad B = \frac{3^n+1}{2} + 3^n + 3V + 1.
\end{align*}
Substituting into Equation \eqref{eq:ind3}, we obtain the result. 
\end{proof}

\begin{example}
Consider, again, the polynomial $r_2(x,t)$ when $t \equiv 1 \pmod 3$. It is enough to consider the polygons where $5 \le \nu_3(a_{2,0}) \le 7$ since larger $3$-adic valuations of $a_{2,0}$ do not affect the number of lattice points under the polygon. In order to compute the area under the polygon, we break the region into one large triangle and $V$ smaller triangles, where $V = \min\{\nu_3(a_{2,0})-5,2\}$. By construction, the smaller triangles have area $3^k/2$, and the large triangle has area $3^n(3^n+1)/4$. See Figure \ref{fig:ind3 proof}.
\end{example}

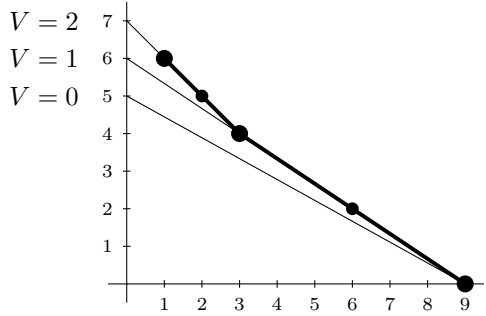
\begin{figure}[!ht]
\centering
\begin{tikzpicture}[scale=.5]
\draw (-.5,0) -- (9.5,0) 
	(0,-.5) -- (0,7.5);
\draw[line width = 1.5pt] (9,0) -- (3,4) -- (1,6);
\foreach \x in {1,...,9}
	\draw (\x,-.1) node[below,font=\tiny] {$\x$} -- (\x,.1);
\foreach \x in {1,...,7}
	\draw (-.1,\x) node[left,font=\tiny] {$\x$} -- (.1,\x);
\foreach \x/\y in {1/6,2/5,3/4,6/2,9/0}
	\draw (\x,\y) node[fill,circle,inner sep=.6mm] {};
\draw (9,0) -- (0,5) (3,4) -- (0,6) (1,6)--(0,7);
\foreach \x/\y in {9/0,3/4,1/6}
	\draw (\x,\y) node[fill,circle,inner sep=.8mm] {};
\foreach \y in {0,1,2}
	\draw (-1,\y+5) node[font=\small,anchor=east] {$V=\y$};
\end{tikzpicture}
\caption{The region under the $(x-1)$-Newton polygon of $r_2(x,t)$.} \label{fig:ind3 proof}
\end{figure}

\subsection{Index calculation: $p \neq 3$}
Recall from Lemma \ref{lem:rmodp} that if $p \mid t^2+t+1$, then $r_n(x,t) \equiv (x-t)^{3^n} \pmod p$. Once again, we may use Taylor expansion to determine the $(x-t)$-development:
\begin{align*} 
r_n(x,t) = \sum_{m=0}^{3^n}a_{n,m}(x-t)^m, \quad \text{ where } \quad a_{n,m} = \frac{r_n^{(m)}(t)}{m!}.
\end{align*}
By Proposition \ref{prop:genrik}, it follows that
\begin{align*}
r_n^{(m)}(t,t) = \frac{3^n!}{(3^n-m)!}\frac{(t-\zeta)(t-\zeta\inv)^{3^n-m}-(t-\zeta\inv)(t-\zeta)^{3^n-m}}{\zeta\inv-\zeta}. \\
\intertext{Hence for $0 \le m < 3^n$,}
a_{n,m} = {3^n\choose m}(t^2+t+1)\frac{(t-\zeta\inv)^{3^n-m-1}-(t-\zeta)^{3^n-m-1}}{\zeta\inv-\zeta}.
\end{align*}
The $p$-adic valuation of these coefficients is
\begin{align} \label{eq:pa_n,m}
\nu_p(a_{n,m}) = \nu_p{3^n \choose m}+\nu_p(t^2+t+1)
\end{align}
since $t \equiv \zeta^{\pm1}\pmod p$ and
\begin{align*}
\frac{(t-\zeta\inv)^{3^n-m-1}-(t-\zeta)^{3^n-m-1}}{\zeta\inv-\zeta} \equiv \pm (\zeta\inv-\zeta)^{3^n-m-2}\equiv \pm(\sqrt{-3})^{3^n-m-2} \not\equiv 0 \pmod p.
\end{align*}
Since $\nu_3(a_{n,m})\ge \nu_3(a_{n,0}) = \nu_3(t^2+t+1)$ for each $0<m<3^n$, it follows that the $(x-t)$-Newton polygon comprised of a single side with endpoints $(0,\nu_p(t^2+t+1))$ and $(3^n,0)$.

\begin{thm} \label{th:indp}
Suppose $p \mid t^2+t+1$, and let $v = \nu_p(t^2+t+1)$. Then
\begin{align*}
\ind_p r_n(x,t) =\frac{1}{2}\Big((3^n-1)(v-1) + \gcd(3^n,v)-1\Big)
\end{align*}
\end{thm}

\begin{proof}
We have just shown that the $(x-t)$-Newton polygon consists of a single side. By definition, the degree of this side is $\gcd(3^n,v)$. It follows from Equation \eqref{eq:pa_n,m} that the residual polynomial is of the form $y^{\gcd(3^n,v)}+c$, as all the intermediary residual coefficients are 0. It is a well known fact that $\disc(y^d+c) = d^dc^{d-1}$. Thus the residual polynomial is separable over $\ff_p$, as $\disc(y^{\gcd(3^n,v)}+c)$ is relatively prime to $p$. The result then follows by Theorem \ref{th:gmn} and an application of Lemma \ref{lem:pick}.
\end{proof}

The formula for the index in Theorem \ref{th:maine index} now follows from Theorems \ref{th:ind3a}, \ref{th:ind3}, and \ref{th:indp}, and the equation for the field discrinint follows from Equations \eqref{eq:disc} and \eqref{eq:disc relation} when $\ell=3$ and the base field is $\qq$.

\section{Proof of Theorem \ref{th:maine decomposition}} \label{sec:primes}
Throughout this section, we assume that $\ell$ is an odd prime, but otherwise we maintain our original notation: $\zeta$ is a $\ell$-th root of unity, $\zeta^+ = \zeta+\zeta\inv$, and $\{K_n;t\}_{n=1}^\infty$ is a specialized tower over $K = \qq(\zeta^+)$. 

Our first goal is to describe the graphs of $\varphi := \varphi(x;\ell)$ over finite fields. If $\ff_q$ is a finite field whose cardinality is congruent to 1 modulo $\ell$, then the graph of $\varphi$ over $\pp\ff_q := \ff_q \cup \{\infty\}$ has an unusually regular structure. (See again Figure \ref{fig:g(127)}.) It is this regularity that will allow us to give the decomposition of any prime of norm $q$ explicitly. We should point out that this method can also be used on a case-by-case basis to determine the decomposition of prime whose norms are not congruent to 1 modulo $\ell$. However, to obtain a theorem similar to Theorem \ref{th:maine decomposition}, one would need to understand how the graph of $\varphi$ over $\pp\ff_q$ sits inside the graph of $\varphi$ over $\pp\ff_q(\zeta)$. The case we consider is just the simplest case, as $\ff_q = \ff_q(\zeta)$ when $q \equiv 1 \pmod \ell$.

We remind the reader that the graphs are constructed to reflect the action of $\varphi$ over the finite field. That is, the graph contains an edge from $a$ to $b$ if and only if $\varphi(a) = b$. In particular, one can recover all the roots of $\varphi^n(x;\ell)-t \pmod p$ by constructing the graph of $\varphi$ over a sufficiently large field of characteristic $p$, then tracing back along all paths of length $n$ that terminate at $t$. This observation translates to the decomposition of primes in the following way.

Suppose $p$ be a prime that does not divide $\ind r_n(x,t;\ell)$, and let $\fr p \subset \ring_K$ be a prime lying over $p$ of norm $q$, i.e. $\ring_K/\fr p \cong \ff_q$. If $q \equiv 1 \pmod \ell$, then our precise understanding of the graphs yields the factorization of $\varphi^n(x;\ell)-t$ over $\ff_q[x]$. From there, we appeal to Dedekind (\cite[Theorem 2.3.9]{c00}), which associates the irreducible factors of $\varphi^n(x;\ell)-t$ to the primes over $\fr p$ in $K_n$.

We now describe the action of $\varphi$ over an arbitrary finite field $\ff$. For any element $a \in \pp\ff-\{\zeta,\zeta\inv\}$, define
\begin{align*}
\beta_a = \frac{a-\zeta}{a-\zeta\inv} \in \ff(\zeta)^\times,
\end{align*}
where $\beta_\infty = 1$. Let $\ord\beta_a$ denote the order of $\beta_a$ in $\ff(\zeta)^\times$, and write $\ord\beta_a = \ell^\lambda d$ where $\gcd(\ell,d) = 1$. Set $\ord_d(\ell)$ to be the order of $\ell$ in $(\zz/d\zz)^\times$. 

\begin{prop} \label{prop:orbit}
If $a \in \pp\ff - \{\zeta,\zeta\inv\}$ and $\ord\beta_a = \ell^\lambda d$, where $\gcd(\ell,d) = 1$, then $\pper(a) = \lambda$ and $\period(\varphi^\lambda(a)) = \ord_d(\ell)$.
\end{prop}

\begin{proof}
Let $a \in \ff$, and suppose $\pper(a) = m$ and $\per(\varphi^m(a)) = n$. Then
\begin{align*}
0 &= \varphi^{m+n}(a) - \varphi^m(a) \\
&= \frac{\zeta\inv(x-\zeta)^{\ell^{m+n}}-\zeta(a-\zeta\inv)^{\ell^{m+n}}}{(a-\zeta)^{\ell^{m+n}}-(a-\zeta\inv)^{\ell^{m+n}}} - \frac{\zeta\inv(a-\zeta)^{\ell^{m}}-\zeta(a-\zeta\inv)^{\ell^{m}}}{(a-\zeta)^{\ell^{m}}-(a-\zeta\inv)^{\ell^{m}}},
\end{align*}
which is equivalent to 
\begin{align*}
(\zeta-\zeta\inv)\Big((a-\zeta)(a-\zeta\inv)\Big)^{\ell^m}\left((a-\zeta)^{\ell^m(\ell^n-1)}-(a-\zeta\inv)^{\ell^m(\ell^n-1)}\right)=0.
\end{align*}
Since $a \not \in\{\zeta,\zeta\inv\}$, we conclude that
\begin{align*}
\left(\frac{a-\zeta}{a-\zeta\inv}\right)^{\ell^m(\ell^n-1)}=1.
\end{align*}
By the minimality of $m$ and $n$, it follows that $m = \lambda$ and $n = \ord_d(\ell)$.
\end{proof}

In particular, if $\ff(\zeta) = \ff$, then the action of $\varphi$ over $\pp\ff$ is completely determined by the orders of the elements in the cyclic group $\ff^\times$.

\begin{thm} \label{th:maine graph}
Let $q$ be a prime power congruent to 1 modulo $\ell$, and write $q-1 = \ell^\lambda\omega$. The graph of $\varphi(x;\ell)$ over $\pp\ff_q$ contains two fixed points, $\zeta$ and $\zeta\inv$, and $\omega$ other periodic points that are arranged into cycles as follows. For each divisor $d$ of $\omega$, there are $\phi(d)$ periodic points of period $\ord_d(\ell)$, where $\phi$ denotes the Euler totient function. Additionally, the graph contains $\omega(\ell-1)\ell^{k-1}$ points of preperiod $k$ for each $1\le k \le \lambda$. 
\end{thm}

\begin{proof}
The map $f \colon (\pp\ff_q-\{\zeta,\zeta\inv\}) \to \ff_q^\times$ defined by $f(a) = \beta_a$ is an isomorphism. (The inverse $f\inv \colon \ff_q^\times \to (\pp\ff_q-\{\zeta,\zeta\inv\})$ is given by $f\inv(\beta) = (\beta\zeta\inv-\zeta)/(\beta-1)$, where $f\inv(1) = \infty$.) Since $\ff_q^\times$ is a cyclic group of order $q-1$, the orders of the $\beta_a$'s are completely determined by the divisors of $q-1$, and the result follows from Proposition \ref{prop:orbit}.
\end{proof}

\begin{remark}
Since the number of preimages cannot exceed the degree of $\varphi$, it follows from a simple point count that if $a \in \pp\ff_q-\{\zeta,\zeta\inv\}$ and $\pper(a) < \lambda$, then the graph contains exactly $\ell$ distinct preimages of $a$. As a result, the branching structure in each component of the graphs is uniform in every direction.
\end{remark}

\begin{example}
By Theorem \ref{th:maine graph}, the graph of $\varphi(x;3)$ over $\pp\ff_{127}$ (shown in Figure \ref{fig:g(127)}) is determined by the divisors of 126, as shown in the table in Figure \ref{fig:table G(3,127)}.
\end{example}

\begin{itsapicture}
\begin{figure}[!ht]
\begin{tabular}{|c|c|c|c|}
\hline
Divisor of 126 & Number of elements & Period & Preperiod \\
$d$ & $\phi(d)$ & $\ord_d(\ell)$ if $\gcd(\ell,d) = 1$ & $\nu_\ell(d)$ \\
\hline
\hline
1 & 1 & 1 & 0 \\
3 & 2 & - & 1 \\
9 & 6 & - & 2 \\
\hline
2 & 1 & 1 & 0 \\
6 & 2 & - & 1 \\
18 & 6 & - & 2 \\
\hline
7 & 6 & 6 & 0 \\
21 & 12 & - & 1 \\
63 & 36 & - & 2 \\
\hline
14 & 6 & 6 & 0 \\
42 & 12 & - & 1 \\
126 & 36 & - & 2 \\
\hline
\end{tabular}
\caption{The cycle structure of $\varphi(x;3)$ over $\pp\ff_{127}$, as determined by Theorem \ref{th:maine graph}.} \label{fig:table G(3,127)}
\end{figure}
\end{itsapicture}

From now on, we assume that $\fr p$ is a prime of characteristic $p$ and norm $q$, where $q \equiv 1 \pmod \ell$. Our next goal is to understand the factorization of $\varphi^n(x;\ell)-t$ in $\ff_q[x]$. By our previous conversation, we know that we can identify all $\ell^n$ roots of $\varphi^n(x)-t \pmod{\fr p}$ by constructing the graph of $\varphi$ over a sufficiently large field, then tracing back along all paths of length $n$ that terminate at $t$. The graph will contain these paths by provided the graph contains vertices with sufficiently large preperiod, and by Theorem \ref{th:maine graph}, this is equivalent to finding $k$ such that $p^k-1$ has sufficiently large $\ell$-adic valuation.

Moreover, in order to keep track of the degrees of the irreducible factors of $\varphi^n(x;\ell)-t$ over $\ff_q[x]$, we assign to each vertex a \emph{$q$-weight}, which we define as
\begin{align*}
\wt_q(b) = [\ff_q(b)\colon \ff_q].
\end{align*}
Thus if $b$ is a root of $\varphi^n(x;\ell)-t$ modulo $\fr p$, then the factorization of $\varphi^n(x;\ell)-t$ over $\ff_q[x]$ contains an irreducible factor of degree $\wt_q(b)$.

\begin{lem} \label{lem:valuation}
Let $m$ be the minimal integer for which $\ell \mid p^m-1$. Then
\begin{align*}
\nu_\ell(p^k-1) = \begin{cases*}\nu_\ell(p^m-1) + \nu_\ell(k) & if $m \mid k$ \\ 0 & otherwise.\end{cases*}
\end{align*}
\end{lem}

\begin{proof}
By definition, $m$ is the order of $p$ in $(\zz/\ell\zz)^\times$, hence $\ell \mid p^k-1$ if and only if $m \mid k$. If $m \mid k$, then $k = rm$ for some integer $r$, so we may write
\begin{align*}
p^k-1 = p^{rm}-1 = (p^m-1)\sum_{i=0}^{r-1} p^{mi}.
\end{align*}
This implies that
\begin{align*}
\nu_\ell(p^k-1) = \nu_\ell(p^m-1) + \nu_\ell\left(\sum_{i=0}^{r-1} p^{mi}\right) = \nu_\ell(p^m-1) + \nu_\ell(r) = \nu_\ell(p^m-1) + \nu_\ell(k),
\end{align*}
where the last equally follows from the fact that $\gcd(m,\ell)=1$, and so $\nu_\ell(r) = \nu_\ell(rm) = \nu_\ell(k)$.
\end{proof}

The take-away is that if we begin by considering the graph of $\varphi$ over $\pp\ff_q$ and we want to extend our graph by adding vertices with larger preperiods, then it suffices to work over a field whose degree over $\ff_q$ is a power of $\ell$. In fact, each time we increase the degree of the extension by $\ell$, the maximal preperiod in the graph is increased by one, thus we can establish a direct relationship be the weight of a vertex and its preperiod. Namely, suppose $a \in \pp\ff_q$ is maximally preperiodic, by which we mean 
\begin{align*}
\pper(a) = \max_{b \in \pp\ff_q}\{\pper(b)\}.
\end{align*} Then for any $b \in \varphi^{-n}(a)$, it follows by Lemma \ref{lem:valuation} that $\wt_q(b) = \ell^n$. In particular, the uniform growth of the graphs established by Theorem \ref{th:maine graph} and Lemma \ref{lem:valuation} means that the factorization of $\varphi^n(x)-t$ in $\ff_q[x]$ is completely determined by the preperiod of $\overline t$, where $\overline t$ denotes the reduction of $t$ modulo $\fr p$. We summarize this result as follows.

\begin{thm} \label{th:maine factorization} 
Let $M = \max_{b \in \pp\ff_q} \{ \pper(b)\}$, and let $\pper(t)$ denote the preperiod of $\overline t$ in $\pp\ff_q$. If $t \equiv \zeta^{\pm1} \pmod{\fr p}$, then $\varphi^n(x;\ell)-t \equiv (x-t)^{\ell^n} \pmod{\fr p}$. Otherwise,
\begin{enumerate}
\item if $n \le M - \pper(t)$, then $\varphi^n(x;\ell)-t$ splits completely into distinct linear factors;
\item if $n > M - \pper(t)$ and $\pper(t) \ge 1$, then $\varphi^n(x;\ell)-t$ factors into $\ell^{M - \pper(t)}$ distinct irreducible factors of degree $\ell^{n-M+\pper(t)}$;
\item if $n > M$ and $\pper(t) = 0$, then $\varphi^n(x;\ell)-t$ factors into $\sum_{i=1}^M \ell^i$ distinct linear factors, and $\ell^M$ distinct irreducible factors of degree $\ell^i$ for each $1 \le i \le n-M$.
\end{enumerate}
\end{thm}

Assuming that $\fr p \nmid (\disc r_n(x,t;\ell))\ring_K$, Theorem \ref{th:maine index} follows immediately from Theorem \ref{th:maine factorization} and an application of Dedekind's theorem on the decomposition of primes \cite[Theorem 2.3.9]{c00}. Namely, the irreducible factors of $\varphi^n(x;\ell)-t$ in $\ff_q[x]$ and their degrees are in a one-to-one correspondence with the prime ideals in $K_n$ lying over $\fr p$ and their inertial degrees, respectively.

We also see, as a consequence of Theorem \ref{th:maine factorization} part (2), that if $\pper(t) = M$, then $\varphi^n(x;\ell)-t$, and hence $r_n(x,t;\ell)$, is irreducible for every $n \ge 1$. Thus the graphs give us a method for identifying values $t \in \ring_K$ that satisfy this irreducibility criterion, i.e. $r_n(x,t;\ell)$ is irreducible for each $n \ge 1$. Specifically, by Dirichlet's theorem on arithmetic progressions, we can always find a prime $\fr p \subset \ring_K$ of norm $q$ congruent to 1 modulo $\ell$. For any of these primes, construct the graph of $\varphi(x;\ell)$ over $\pp\ff_q$, then every maximally preperiodic element in the graph gives an equivalence class of values modulo $\fr p$ which yield irreducible iterates.

\begin{example} \label{ex:G(5,31)}
Fix a primitive 5-th root of unity $\zeta_5$. By cyclotomic reciprocity, $\qq(\zeta_5^+)$ contains two primes of norm 31, which we represent by $\fr p_1 = (\zeta_5^+-18)$ and $\fr p_2 = (\zeta_5^+-12)$. By Theorem \ref{th:maine graph} the structure of the graph of $\varphi(x;5)$ over $\pp\ff_{31}$ is determined by the divisors of 30; see Figure \ref{fig:table g(31)}.

\begin{itsapicture}
\begin{figure}[!ht]
\centering
\begin{tabular}{|c|c|c|c|}
\hline
Divisor of 30 & Number of elements & Period & Preperiod \\
$d$ & $\phi(d)$ & $\ord_d(\ell)$ if $\gcd(\ell,d) = 1$ & $\nu_\ell(d)$ \\
\hline
\hline
1 & 1 & 1 & 0 \\
5 & 4 & - & 1 \\
\hline
2 & 1 & 1 & 0 \\
10 & 4 & - & 1 \\
\hline
3 & 2 & 2 & 0 \\
15 & 8 & - & 1 \\
\hline
6 & 2 & 2 & 0 \\
30 & 8 & - & 1 \\
\hline
\end{tabular}
\caption{The graph of $\varphi(x;5)$ over $\pp\ff_{31}$ is determined by the divisors of 30.} \label{fig:table g(31)}
\end{figure}
\end{itsapicture}

Although the structures of the graphs are identical, we remind the reader that the decomposition of $\fr p_i$ is determined by $t \bmod{\fr p_i}$, and so the behavior of the primes may differ. For each prime, we can determine $\zeta_5 \bmod{\fr p_i}$ by considering the factorization of $x^2-\zeta_5^+x+1$ in $\ff_{31}[x]$. Here, $x^2-\zeta_5^+x + 1 \equiv (x-2)(x+15) \pmod{\fr p_1}$, and $x^2-\zeta_5^+x + 1 \equiv (x-4)(x-8) \pmod{\fr p_2}$. The graphs of $\varphi(x;5)$ are shown in Figures \ref{fig:gp1} and \ref{fig:gp2}, respectively. We see, for example, that if $t \in \zz$ and $t \equiv 10 \pmod{31}$, then $\fr p_1$ and $\fr p_2$ are totally inert in any tower specialized at $t$. On the other hand, if $t \equiv 14 \pmod{31}$, then $\fr p_2$ is totally inert, but $\fr p_1$ is not.

\begin{itsapicture}
\begin{figure}[!ht]
\centering
\begin{tikzpicture}[>=latex]
\draw (-3,0) circle (12pt) node[circle,inner sep=3.5mm] (2){} node {$2$};
\draw[->] (2) edge[loop left] (2);
\draw (-3,-4) circle (12pt) node[circle,inner sep=3.5mm] (-15) {} node{$-15$};
\draw[->] (-15) edge[loop left] (-15);
\draw (0,0) circle (12pt) node[circle,inner sep=3.5mm] (infty){} node {$\infty$};
\draw[->] (infty) edge[loop left] (infty);
\foreach \x/\y in {0/0,-1/1,-12/2,-13/3} {
	\draw (0,0)++(22.5-90+45*\y:1.5) circle (12pt) node[circle,inner sep=3.5mm] (\x){} node {$\x$};
	\draw[->] (\x) edge (infty);
	}
\draw (0,-4) circle (12pt) node[circle,inner sep=3.5mm] (9){} node {$9$};
\draw[->] (9) edge[loop left] (9);
\foreach \x/\y in {-14/0,1/1,7/2,11/3} {
	\draw (0,-4)++(22.5-90+45*\y:1.5) circle (12pt) node[circle,inner sep=3.5mm] (\x){} node {$\x$};
	\draw[->] (\x) edge (9);
	}
\draw (4.5,0) circle (12pt) node[circle,inner sep=3.5mm] (4){} node {$4$}
	(6,0) circle (12pt) node[circle,inner sep=3.5mm] (14){} node {$14$};
\draw[->] (4) edge[bend left] (14) (14) edge[bend left] (4);
\foreach \x/\y in {-3/0,-2/1,3/2,8/3} {
	\draw (4.5,0)++(22.5+90+45*\y:1.5) circle (12pt) node[circle,inner sep=3.5mm] (\x){} node {$\x$};
	\draw[->] (\x) edge (4);
	}
\foreach \x/\y in {-11/0,-10/1,10/2,15/3} {
	\draw (6,0)++(22.5-90+45*\y:1.5) circle (12pt) node[circle,inner sep=3.5mm] (\x){} node {$\x$};
	\draw[->] (\x) edge (14);
	}
\draw (4.5,-4) circle (12pt) node[circle,inner sep=3.5mm] (-6){} node {$-6$}
	(6,-4) circle (12pt) node[circle,inner sep=3.5mm] (-7){} node {$-7$};
\draw[->] (-6) edge[bend left] (-7) (-7) edge[bend left] (-6);
\foreach \x/\y in {-5/0,-4/1,5/2,12/3} {
	\draw (4.5,-4)++(22.5+90+45*\y:1.5) circle (12pt) node[circle,inner sep=3.5mm] (\x){} node {$\x$};
	\draw[->] (\x) edge (-6);
	}
\foreach \x/\y in {-9/0,-8/1,6/2,13/3} {
	\draw (6,-4)++(22.5-90+45*\y:1.5) circle (12pt) node[circle,inner sep=3.5mm] (\x){} node {$\x$};
	\draw[->] (\x) edge (-7);
	}
\end{tikzpicture}
\caption{The graph of $\varphi(x;5)$ over $\qq(\zeta_5^+)/\fr p_1 \cong \pp\ff_{31}$.} \label{fig:gp1}
\end{figure}
\end{itsapicture}

\begin{itsapicture}
\begin{figure}[!ht]
\centering
\begin{tikzpicture}[>=latex]
\draw (-3,0) circle (12pt) node[circle,inner sep=3.5mm] (2){} node {$4$};
\draw[->] (2) edge[loop left] (2);
\draw (-3,-4) circle (12pt) node[circle,inner sep=3.5mm] (-15) {} node{$8$};
\draw[->] (-15) edge[loop left] (-15);
\draw (0,0) circle (12pt) node[circle,inner sep=3.5mm] (infty){} node {$\infty$};
\draw[->] (infty) edge[loop left] (infty);
\foreach \x/\y in {0/0,-1/1,12/2,13/3} {
	\draw (0,0)++(22.5-90+45*\y:1.5) circle (12pt) node[circle,inner sep=3.5mm] (\x){} node {$\x$};
	\draw[->] (\x) edge (infty);
	}
\draw (0,-4) circle (12pt) node[circle,inner sep=3.5mm] (9){} node {$6$};
\draw[->] (9) edge[loop left] (9);
\foreach \x/\y in {1/0,11/1,-14/2,-5/3} {
	\draw (0,-4)++(22.5-90+45*\y:1.5) circle (12pt) node[circle,inner sep=3.5mm] (\x){} node {$\x$};
	\draw[->] (\x) edge (9);
	}
\draw (4.5,0) circle (12pt) node[circle,inner sep=3.5mm] (4){} node {$3$}
	(6,0) circle (12pt) node[circle,inner sep=3.5mm] (14){} node {$9$};
\draw[->] (4) edge[bend left] (14) (14) edge[bend left] (4);
\foreach \x/\y in {5/0,-13/1,-10/2,-7/3} {
	\draw (4.5,0)++(22.5+90+45*\y:1.5) circle (12pt) node[circle,inner sep=3.5mm] (\x){} node {$\x$};
	\draw[->] (\x) edge (4);
	}
\foreach \x/\y in {7/0,-12/1,-9/2,-6/3} {
	\draw (6,0)++(22.5-90+45*\y:1.5) circle (12pt) node[circle,inner sep=3.5mm] (\x){} node {$\x$};
	\draw[->] (\x) edge (14);
	}
\draw (4.5,-4) circle (12pt) node[circle,inner sep=3.5mm] (-6){} node {$-3$}
	(6,-4) circle (12pt) node[circle,inner sep=3.5mm] (-7){} node {$15$};
\draw[->] (-6) edge[bend left] (-7) (-7) edge[bend left] (-6);
\foreach \x/\y in {10/0,14/1,-15/2,-8/3} {
	\draw (4.5,-4)++(22.5+90+45*\y:1.5) circle (12pt) node[circle,inner sep=3.5mm] (\x){} node {$\x$};
	\draw[->] (\x) edge (-6);
	}
\foreach \x/\y in {2/0,-11/1,-4/2,-2/3} {
	\draw (6,-4)++(22.5-90+45*\y:1.5) circle (12pt) node[circle,inner sep=3.5mm] (\x){} node {$\x$};
	\draw[->] (\x) edge (-7);
	}
\end{tikzpicture}
\caption{The graph of $\varphi(x;5)$ over $\qq(\zeta_5^+)/\fr p_2\cong \pp\ff_{31}$.} \label{fig:gp2}
\end{figure}
\end{itsapicture}

\end{example}

\bibliographystyle{abbrv}
\bibliography{DiscOfRikunaExt}
\end{document}